\documentclass[11pt,a4paper]{article}

\usepackage{graphicx,latexsym,euscript,makeidx,color,bm}
\usepackage{amsmath,amsfonts,amssymb,amsthm,thmtools,mathtools,mathrsfs,enumerate}
\usepackage[colorlinks,linkcolor=blue,anchorcolor=green,citecolor=red]{hyperref}
\usepackage{bm}

\usepackage{tikz}
\usetikzlibrary{fit,calc,positioning} 
\tikzstyle{Block}=[rectangle,minimum width=3cm,minimum height=1cm,text centered,text width=5.2cm,draw=black]
\tikzstyle{Implication}=[rectangle,minimum width=3cm,minimum height=1cm,text centered,text width=5.2cm]
\tikzstyle{jian}=[<->, >=stealth]
\usetikzlibrary{positioning, arrows.meta, calc}

\usepackage{geometry}
\allowdisplaybreaks[4]

  \def\cA{{\cal A}}   \def\bA{\bar{A}} \def\hA{\widehat{A}} 
  \def\cB{{\cal B}}   \def\bB{\bar{B}} \def\hB{\widehat{B}} 
\def\dbC{\mathbb{C}}     \def\bC{\bar{C}} \def\hC{\widehat{C}} 
      
\def\dbE{\mathbb{E}}    
\def\dbF{\mathbb{F}} \def\sF{\mathscr{F}}

  \def\cK{{\cal K}}  
  \def\cL{{\cal L}}  
    
\def\dbN{\mathbb{N}}  \def\cN{{\cal N}}  
  \def\cO{{\cal O}}  
\def\dbP{\mathbb{P}}    
    
\def\dbR{\mathbb{R}}  \def\cR{{\cal R}}  
\def\dbS{\mathbb{S}}    
    
\def\dbU{\mathbb{U}}

\def\lt{\left}
\def\hb{\hbox}
\def\ms{\medskip}
\def\rt{\right}

        \def\lan{\langle}    \def\as{\hb{a.s.}}
   \def\ran{\rangle}    \def\tr{\hb{tr$\,$}}
         
\def\no{\noindent}          
\def\hp{\hphantom}         
         
\def\rf{\eqref}            
\def\cd{\cdot}             
\def\deq{\triangleq}     \def\({\Big (}       \def\ba{\begin{aligned}}
\def\les{\leqslant}      \def\){\Big )}       \def\ea{\end{aligned}}
\def\ges{\geqslant}      \def\[{\Big[}        \def\bel{\begin{equation}\label}
          \def\]{\Big]}        \def\ee{\end{equation}}
      \def\q{\quad}        
         \def\qq{\qquad}      

  \def\G{\varGamma}      \def\Om{\varOmega}  
\def\b{\beta}   \def\D{\varDelta}   \def\d{\delta}   \def\F{\varPhi}     
\def\z{\zeta}     \def\th{\theta}    \def\si{\sigma}
   \def\l{\lambda}           
    \def\i{\infty}         
\def\bp{\begin{pmatrix}}
\def\ep{\end{pmatrix}}

       \def\lt{\left}          \def\hb{\hbox}
\def\ms{\medskip}     \def\rt{\right}         
          \def\lan{\langle}       \def\as{\text{a.s.}}
\def\q{\quad}         \def\ran{\rangle}       \def\tr{\hb{tr$\,$}}
\def\qq{\qquad}             
\def\no{\noindent}          
\def\hp{\hphantom}         
         
\def\rf{\eqref}        
\def\cd{\cdot}         
\def\deq{\triangleq}  \def\({\Big(}           \def\les{\leqslant}
       \def\){\Big)}           \def\ges{\geqslant}
   \def\[{\Big[}           
     \def\]{\Big]}

\DeclareMathOperator{\im}{im}

\DeclareMathOperator{\rank}{rank}
\newtheoremstyle{thry}
{}      
{}      
{\sl}   
{}      
{\bf}   
{.}     
{.5em}  
{}      

\theoremstyle{thry}

\newtheorem{theorem}{Theorem}[section]
\newtheorem{proposition}[theorem]{Proposition}
\newtheorem{corollary}[theorem]{Corollary}
\newtheorem{lemma}[theorem]{Lemma}

\theoremstyle{definition}

\newtheorem{example}[theorem]{Example}

\theoremstyle{remark}
\newtheorem{remark}[theorem]{Remark}

\makeatletter
   
   \@addtoreset{equation}{section}
   \newcommand{\setword}[2]{%
   \phantomsection
   #1\def\@currentlabel{\unexpanded{#1}}\label{#2}%
   }
\makeatother



\begin{document} 

\title{\bf Exact Controllability of Backward-Structured Mean-Field SDEs: 
Hautus Criteria and Initial-Time Dichotomy} 

\author{
Jingrui Sun\thanks{
Department of Mathematics and SUSTech International Center for Mathematics,
Southern University of Science and Technology, 
Shenzhen, Guangdong, 518055, China (Email: sunjr@sustech.edu.cn).
This author is supported by NSFC grants 12322118 and 12271242, 
and by Shenzhen Science and Technology Program grant JCYJ20250604144337051.}
~~~
Lvning Yuan\thanks{School of Mathematics and Computing Science,
Guangxi Colleges and Universities Key Laboratory of Data Analysis and
Computation, Guilin University of Electronic Technology, Guilin, 541004,
China (Email: yuanln@guet.edu.cn).}
~~~
Xurun Zuo\thanks{Corresponding author. Department of Mathematics, Southern University of Science and Technology, 
Shenzhen, 518055, China (Email: 12331006@mail.sustech.edu.cn).}
}

\maketitle

\no{\bf Abstract.}
This paper studies exact controllability of linear mean-field stochastic 
differential equations with backward-structure. 
We establish Hautus criteria and uncover a sharp dichotomy between the 
zero initial time and positive initial times. 
At time zero, exact controllability is characterized by a Hautus condition 
for the mean dynamics, with additional control directions generated by the 
centered stochastic dynamics. 
At any positive initial time, exact controllability further requires a 
stochastic Hautus condition for the centered system. 
Unlike the classical eigenvector-based condition, this stochastic criterion 
is formulated in terms of positive-semidefinite eigenmatrices of an associated 
Lyapunov-type operator. 
The dichotomy arises from the triviality of the initial sigma-field. 
We also clarify the controllability relations among the corresponding ODE, SDE, 
and mean-field SDE systems.

\ms
\no{\bf Key words.}
Backward-structured mean-field SDEs, exact controllability, Hautus criteria, 
initial-time dichotomy, controllability Gramian, positive-semidefinite eigenmatrices.

\ms
\no{\bf MSC codes.} 93B05, 93E03, 60H10.

\section{Introduction}\label{Sec:Intro}   

Let $(\Om,\sF,\dbP)$ be a complete probability space on which a standard one-dimensional
Brownian motion $W=\{W(t);t\ges0\}$ is defined, and let $\dbF=\{\sF_t\}_{t\ges0}$ be
the usual augmentation of the natural filtration generated by $W$.
For $0\les t<T$, denote by $L_{\sF_t}^2(\Om;\dbR^n)$ the space of $\sF_t$-measurable, 
square-integrable $\dbR^n$-valued random variables, and by $L_{\dbF}^2(t,T;\dbR^m)$ the 
space of $\dbF$-progressively measurable, square-integrable $\dbR^m$-valued processes 
on $[t,T]$. Consider the following controlled linear mean-field stochastic differential 
equation (MF-SDE, for short):
\begin{equation}\label{eq:system}\left\{\begin{aligned}
dX(s)&=[AX(s)+\bA\dbE X(s)+Bu(s)+\bB\dbE u(s) \\
     &\hp{=\ } +Cz(s)+\bC\dbE z(s)]ds+z(s)dW(s), \q s\in[t,T],\\
 X(t)&=\eta,
\end{aligned}\right.\end{equation}
where the coefficients 
$$
A,\bA,C,\bC\in\dbR^{n\times n}, \q B,\bB\in\dbR^{n\times m}
$$
are constant matrices. 
Denote the solution of system \rf{eq:system} by $ X(\cd\,;t,\eta,u,z)$. 
We say that system \eqref{eq:system} is {\it exactly controllable} 
over $[t,T]$ if, for every 
$$
\eta\in L^2_{\sF_t}(\Om;\dbR^n) \q\hbox{and}\q \xi\in L^2_{\sF_T}(\Om;\dbR^n),
$$
there exists a control pair $(u(\cd),z(\cd))\in L_{\dbF}^2(t,T;\dbR^m)\times L_{\dbF}^2(t,T;\dbR^n)$ 
such that the corresponding solution satisfies
$$
X(T;t,\eta,u,z)=\xi,\q\as
$$
We call $(t,\eta)$ the {\it initial pair}, $(T,\xi)$ the {\it terminal pair}, and $(\eta,\xi)$ the {\it initial-terminal pair}.

\ms 

A distinctive feature of this formulation is that the role of the initial time is not merely
to determine the length of the control interval. Indeed, the state can be decomposed into its
mean and centered parts:
$$
X(s)=\dbE X(s)+[X(s)-\dbE X(s)], \q s\in[t,T].
$$
Under the usual augmentation of the Brownian filtration, $\sF_0$ is trivial, and hence every
$\sF_0$-measurable square-integrable initial state is deterministic. In contrast, when $t>0$,
the initial state may have an arbitrary centered random component. This raises a natural
question: are the exact controllability conditions at the zero initial time and at positive
initial times the same? One of the main findings of this paper is that they are not. The two
cases are governed by different Hautus conditions, leading to a sharp initial-time dichotomy.

\ms 

For deterministic linear systems, controllability is classically characterized by the Kalman
rank condition and the Hautus criterion \cite{hautus}. In stochastic systems, exact
controllability is more subtle because the target is random, the controls must be adapted,
and the diffusion term may generate additional reachable directions. The exact controllability
of backward-structured linear SDEs was initiated by Peng \cite{peng1994}, who obtained a
Kalman-type rank condition for systems with constant coefficients. Liu and Peng
\cite{liu2010} extended the analysis to bounded time-varying deterministic coefficients.
Wang, Yang, Yong, and Yu \cite{wang2017} introduced $L^p$-exact controllability for linear
SDEs with random coefficients and established BSDE-based observability criteria together with
several equivalent characterizations of exact and null controllability. Bi, Sun, and Xiong
\cite{bi2020} characterized the controllability of prescribed state pairs. 
Related problems
have also been studied through controllability operators and stochastic reachability
\cite{ehrhardt1982,mahmudov2000,zabczyk1981}.
Furthermore, Buckdahn, Quincampoix, and Tessitore \cite{buckdahn2006} characterized approximate controllability in terms of dual BSDEs and invariant subspaces, while Goreac \cite{goreac2008} obtained a Kalman-type condition.
Wang and Yu \cite{wang2020} studied partial controllability of SDEs together with exact
controllability of FBSDEs. 
More recently, Sun \cite{sun2026} established a
Hautus criterion for the exact controllability of the linear SDE
\bel{eq=0}\left\{\begin{aligned}
dX(s) &= [AX(s)+Bu(s)+Cz(s)]ds+z(s)dW(s),\\
 X(0) &= x\in\dbR^n.
\end{aligned}\right.\ee
A notable feature of that criterion is that the stochastic obstruction is described 
by positive-semidefinite eigenmatrices of a Lyapunov-type operator, rather than by 
eigenvectors of the drift matrix alone.

\ms  

Controllability of linear mean-field stochastic systems has also attracted considerable attention. 
Goreac \cite{goreac2014} studied approximate controllability and obtained necessary conditions for 
exact controllability. 
For constant coefficients, Yu \cite{yu2021} established controllability Gramian and Kalman-type rank 
characterizations for system \eqref{eq:system}.
Ye and Yu \cite{ye2020} investigated exact controllability for linear mean-field stochastic
systems with time-varying random coefficients. 
Chen and Yu \cite{chen2024} studied exact controllability for mean-field type linear game-based 
control systems and obtained Gramian-type and Kalman-type criteria. 
Goreac, Li, and Zhang \cite{goreac2026} studied controllability of linear mean-field stochastic systems with reduced-rank coefficients, obtaining exact controllability results with less regular controls and introducing exact terminal controllability to Gaussian laws.
Controllability of mean-field systems has also been investigated for stochastic evolution equations with delay \cite{mahmudov2006} and stochastic integrodifferential equations in Hilbert spaces \cite{park2008}, as well as at the level of terminal probability distributions through nonlinear Fokker–Planck equations \cite{barbu2023}.
These works provide important operator, Gramian, and rank characterizations, as well as sufficient controllability conditions for stochastic evolution equations.
However, to the best of our knowledge, no Hautus criterion has been established for 
backward-structured linear mean-field stochastic systems. 
In particular, the effect of the initial sigma-field on exact controllability has not been isolated.

\ms 

The purpose of this paper is to establish Hautus criteria for system \eqref{eq:system} and to identify the precise difference between the zero initial time and positive initial times. Set
$$
\hA=A+\bA, \q \hB=B+\bB, \q \hC=C+\bC.
$$
The main results are summarized as follows.

\ms 

(i) We first study exact controllability at the zero initial time. 
Our main result is the Hautus criterion established in \autoref{thm:criterion-U}. 
Let $P$ be any full-column-rank matrix whose columns form a basis of the smallest 
subspace containing $\im B$ and invariant under both $A$ and $C$. 
Then exact controllability is equivalent to
$$
\hA^{\,\top} x=\l x,\q \hB^\top x=0, \q P^\top\hC^\top x=0  \q\Longrightarrow\q  x=0,
$$
for every $\l\in\dbC$ and $x\in\dbC^n$. 
By the classical Hautus criterion, this spectral condition is equivalent to controllability 
of the augmented deterministic pair
$$
\big[\hA,(\hB,\hC P)\big].
$$
Thus, through the term $\hC P$, the reachable directions of the centered stochastic system 
provide additional control directions for the mean dynamics. 
Consequently, the mean-field system may be exactly controllable at time zero even when the 
centered stochastic system itself is not.

\ms

(ii) We next study exact controllability at positive initial times. 
Our main result is the Hautus criterion established in \autoref{thm:controllability-t>0}. 
It shows that exact controllability on one, and hence every, interval $[t,T]$ with $t>0$ 
is equivalent to the simultaneous validity of two spectral conditions. First,
$$
B^\top H\neq0
$$
for every nonzero positive-semidefinite eigenmatrix $H$ of the Lyapunov-type operator
$$
\cL_{(-A,C)}(H) = -HA-A^\top H+C^\top HC.
$$
Second, for every $\l\in\dbC$ and $x\in\dbC^n$,
$$
\hA^{\,\top} x=\l x, \q \hB^\top x=0, \q \hC^\top x=0 \q\Longrightarrow\q x=0. 
$$
Thus, the positive-time criterion combines a stochastic Hautus condition formulated 
in terms of positive-semidefinite eigenmatrices with a classical eigenvector condition 
for the mean dynamics. 

\ms 

Comparison with the zero-time criterion reveals a sharp initial-time dichotomy. 
At time zero, the reachable directions of the centered system only provide additional 
effective control directions for the mean dynamics. 
At any positive initial time, however, the centered stochastic system must itself be 
exactly controllable. 
This distinction results from the triviality of $\sF_0$, rather than from the length 
of the control interval.

\ms

(iii) Finally, we establish several supporting characterizations and compare the controllability 
properties of the associated systems. 
At the zero initial time, we characterize the transferability of prescribed initial--terminal 
state pairs, recover the known Gramian criterion, and give a concise proof of the finite-dimensional 
Kalman-type condition; see \autoref{thm:controllability-characterization}, 
\autoref{coro:G(T)-controllability}, and \autoref{thm:G(t,T)}. 
At positive initial times, we derive equivalent formulations in terms of zero initial or 
terminal states, together with the corresponding operator and Gramian characterizations; 
see \autoref{prop:equivalence-controllability} and \autoref{thm:MF-controllability-t-T}. 
We further show that, unlike the mean-field system, exact controllability of the corresponding 
non-mean-field SDE is independent of whether the initial time is zero or positive; 
see \autoref{prop:SDE-controllability}. 
Finally, we determine the exact controllability relationships among the MF-SDE \eqref{eq:system}, 
the corresponding SDE \eqref{eq:system-SDE}, and the associated ODE \eqref{eq:system-ODE}; 
see \autoref{coro:hA-hB}, \autoref{coro:MF-SDE-t>0-SDE}, and \autoref{thm:relation}. 
The valid implications and the failure of several natural converses are illustrated by explicit 
examples and summarized in Figure \ref{fig:summary}.

\ms

The rest of the paper is organized as follows. 
In \autoref{sec:t=0}, we study exact controllability at the zero initial time, beginning with 
transferability and Gramian characterizations and then deriving the finite-dimensional and Hautus 
criteria. In \autoref{sec:controllability-t>0}, we treat positive initial times, establish the 
operator and Gramian characterizations, and prove the second Hautus criterion together with the 
initial-time dichotomy. Finally, \autoref{sec:relationship} compares the exact controllability 
properties of the associated ODE, SDE, and MF-SDE systems.

\section{Exact Controllability at the Zero Initial Time}\label{sec:t=0}

Throughout the paper, $I_n$ denotes the $n\times n$ identity matrix. 
We use $M^\top$, $\tr(M)$, and $\im M$ for the transpose, trace, and range of a
matrix $M$, respectively.   
For a subspace $V\subset\dbR^n$, $V^\perp$ denotes its orthogonal complement.
We write $\dbS^n$ for the space of symmetric real $n\times n$ matrices,
$\dbS_+^n$ for its positive-definite cone, and $\bar\dbS_+^n$ for its
positive-semidefinite cone. The symbols $\dbN$ and $\dbN_+$ denote
$\{0,1,2,\ldots\}$ and $\{1,2,\ldots\}$, respectively.
For $0\les t<T$, set
$$
\dbU[t,T]\deq L_{\dbF}^2(t,T;\dbR^m)\times L_{\dbF}^2(t,T;\dbR^n),
$$
equipped with the usual product Hilbert space structure.
Standard well-posedness results for linear MF-SDEs and MF-BSDEs ensure that
all equations considered below admit unique square-integrable adapted
solutions. The letter $K$ denotes a generic positive constant that may vary
from line to line.

\ms 

On $\dbR^{n\times m}$, we use the Frobenius inner product and norm
$$
\lan M,N\ran\deq\tr(M^\top N),
\q
|M|\deq\lan M,M\ran^{1/2}.
$$
The same notation $\lan\cd,\cd\ran$ and $|\cd|$ is used for the inner
product and induced norm on all Hilbert spaces under consideration, with
the precise meaning determined by the context. In particular, the inner
products on $L_{\sF_t}^2(\Om;\dbR^n)$ and
$L_{\dbF}^2(t,T;\dbR^m)$ are
$$
\dbE\lan\xi,\eta\ran
\q\hbox{and}\q
\dbE\int_t^T\lan u(s),v(s)\ran ds,
$$
respectively. Finally, for a bounded linear operator $\cA$ between Hilbert
spaces, $\cA^*$ denotes its adjoint.

\subsection{Transferability and the Gramian characterization}

We first study exact controllability of system \eqref{eq:system} with initial time $t=0$.
Let $\F_0(\cd)=\{\F_0(t);t\ges0\}$ be the solution of the following matrix-valued linear MF-SDE:
\bel{eq:Phi}\lt\{\begin{aligned}
d\F_0(t)&= -\big\{\F_0(t)A+[\dbE\F_0(t)]\bA\big\}dt
           -\big\{\F_0(t)C+[\dbE\F_0(t)]\bC\big\}dW(t), \q t\ges0, \\
 \F_0(0)&= I_n.
\end{aligned}\rt.\ee
Standard estimates for linear matrix SDEs yield
$$
\dbE\sup_{0\les s\les T}|\F_0(s)|^2<\i, \q\forall T>0,
$$
and taking expectations in \eqref{eq:Phi} gives
$$
\dbE\F_0(s)=e^{-s\hA},\q s\in[0,T],
$$
which is invertible. Based on $\F_0(\cd)$, we define a symmetric matrix
\begin{equation}\label{eq:Gramian}
G(0,T)\deq\dbE\int_0^T\(\F_0(t)B+\dbE[\F_0(t)]\bB\)
                    \(B^{\top}\F_0(t)^{\top}+\bB^\top\dbE[\F_0(t)^\top]\)dt.
\end{equation}
The matrix \eqref{eq:Gramian} can be regarded as a mean-field controllability Gramian.

\ms

The following duality identity, expressed in terms of the fundamental matrix $\F_0(\cd)$, 
will be used to characterize the transferability of prescribed initial--terminal state 
pairs and to derive the associated controllability Gramian.
 
\begin{proposition}\label{prop:EPhiY}
Let $f(\cd)\in L_{\dbF}^2(0,T;\dbR^n)$ and $\xi\in L_{\sF_T}^2(\Om;\dbR^n)$.
Then the adapted solution $(Y(\cd),Z(\cd))$ of the mean-field backward SDE
\begin{equation}\label{eq:EPhiY}\lt\{\begin{aligned}
dY(t) &=[AY(t)+\bA\dbE Y(t)+CZ(t)+\bC\dbE Z(t)+f(t)]dt+Z(t)dW(t),\\
 Y(T) &=\xi
\end{aligned}\rt.\end{equation}
satisfies
$$
\dbE[\F_0(t)Y(t)]=\dbE\lt[\F_0(T)\xi-\int_t^T\F_0(s)f(s)ds\rt],\q\forall t\in[0,T],
$$
where $\F_0(\cd)$ is the solution of \eqref{eq:Phi}.
\end{proposition}

\begin{proof}
Applying It\^o's formula to $\F_0Y$ yields
$$
d(\F_0Y)=[-(\dbE\F_0)\bA Y+\F_0\bA\dbE Y+\F_0\bC\dbE Z-(\dbE\F_0)\bC Z+\F_0f]dt+\G dW,
$$
where
$$
\G(t)\deq\F_0(t)[Z(t)-CY(t)]-[\dbE\F_0(t)]\bC Y(t).
$$
We first justify that the stochastic integral has zero expectation. 
By the standard estimates for $\F_0$, $Y$, and $Z$, together with the Cauchy--Schwarz inequality,
\begin{align*}
\dbE\lt(\int_0^T|\G(s)|^2ds\rt)^{1/2}
&\les \dbE\lt[\sup_{0\les s\les T}|\F_0(s)|\lt(\int_0^T|Z(s)-CY(s)|^2ds\rt)^{1/2}\rt]  \\
&\hp{=\ } +\sup_{0\les s\les T}\big|[\dbE\F_0(s)]\bC\big|\,\dbE\lt(\int_0^T|Y(s)|^2ds\rt)^{1/2}  \\
&\les\lt(\dbE\sup_{0\les s\les T}|\F_0(s)|^2\rt)^{1/2}\lt(\dbE\int_0^T|Z(s)-CY(s)|^2ds\rt)^{1/2} \\
&\hp{=\ } +K\lt(\dbE\sup_{0\les s\les T}|Y(s)|^2\rt)^{1/2}<\i.
\end{align*}
Hence,
$$
M(t)\deq\int_0^t\G(s)dW(s),\q t\in[0,T],
$$
is a martingale. Integrating the It\^o identity from $t$ to $T$ and taking expectations, 
we obtain
\begin{align*}
\dbE[\F_0(T)\xi]-\dbE[\F_0(t)Y(t)]
&=\dbE\int_t^T\[-(\dbE\F_0)\bA Y+\F_0\bA\dbE Y +\F_0\bC\dbE Z-(\dbE\F_0)\bC Z+\F_0f\]ds.
\end{align*}
Since $\dbE\F_0$, $\dbE Y$, and $\dbE Z$ are deterministic,
\begin{align*}
\dbE\big[-(\dbE\F_0)\bA Y+\F_0\bA\dbE Y\big] &= -(\dbE\F_0)\bA\dbE Y+(\dbE\F_0)\bA\dbE Y=0, \\
\dbE\big[\F_0\bC\dbE Z-(\dbE\F_0)\bC Z\big]  &=  (\dbE\F_0)\bC\dbE Z-(\dbE\F_0)\bC\dbE Z=0.
\end{align*}
Therefore,
$$
\dbE[\F_0(T)\xi]-\dbE[\F_0(t)Y(t)]=\dbE\int_t^T\F_0(s)f(s)ds, \q t\in[0,T],
$$
which proves the desired result.  
\end{proof}

Throughout this section, the initial time is fixed at $t=0$. 
Accordingly, for $x\in\dbR^n$ and a control pair $(u(\cd),z(\cd))\in\dbU[0,T]$, 
we write
$$
X(\cd\,;x,u,z)\deq X(\cd\,;0,x,u,z),
$$
where $X(\cd\,;0,x,u,z)$ denotes the solution of the state equation \rf{eq:system} 
with initial condition $X(0)=x$ and control processes $(u,z)$. 
The following theorem characterizes when a prescribed initial--terminal 
state pair $(x,\xi)$ can be connected at the zero initial time.

\begin{theorem}\label{thm:controllability-characterization}
Let $(x,\xi)\in\dbR^n\times L_{\sF_T}^2(\Om;\dbR^n)$.
Then there exists a control pair $(u(\cd),z(\cd))\in\dbU[0,T]$ such that
$$
X(T;x,u,z)=\xi,\q\as
$$
if and only if
$$
x-\dbE[\F_0(T)\xi]\in\im G(0,T),
$$
where $\F_0(\cd)$ and $G(0,T)$ are defined by \rf{eq:Phi} and \rf{eq:Gramian}, respectively.
\end{theorem}

\begin{proof}
Define the bounded linear operator
$$
\cR:L_{\dbF}^2(0,T;\dbR^m)\longrightarrow\dbR^n
$$
by
$$
\cR u\deq-\dbE\int_0^T\F_0(s)\big[B u(s)+\bB\dbE u(s)\big]ds.
$$
Fix $u(\cd)\in L_{\dbF}^2(0,T;\dbR^m)$, and let $(Y(\cd),Z(\cd))$ be the adapted solution of
$$
\lt\{\begin{aligned}
dY(s) &=[AY(s)+\bA\dbE Y(s)+Bu(s)+\bB\dbE u(s)+CZ(s)+\bC\dbE Z(s)]ds+Z(s)dW(s),\\
Y(T)&=\xi.
\end{aligned}\rt.
$$
Since $\F_0(0)=I_n$ and $\sF_0$ is trivial, \autoref{prop:EPhiY} gives
$$
Y(0)=\dbE[\F_0(T)\xi]+\cR u.
$$
Consequently, the pair $(x,\xi)$ can be connected if and only if
$$
x-\dbE[\F_0(T)\xi]\in\im\cR.
$$
For $a\in\dbR^n$, a direct calculation shows that
$$
(\cR^*a)(s)=-\[B^\top\F_0(s)^\top+\bB^\top[\dbE\F_0(s)]^\top\]a.
$$ 
Hence, by the definition of $G(0,T)$,
$$
\cR\cR^*a=G(0,T)a,\q a\in\dbR^n.
$$
Thus, $G(0,T)=\cR\cR^*$. Moreover, $\ker(\cR\cR^*)=\ker\cR^*$.
Since the range of $\cR$ is a subspace of the finite-dimensional
space $\dbR^n$, it is closed. Therefore,
$$
\im G(0,T)=\im(\cR\cR^*)=(\ker\cR^*)^\perp=\im\cR.
$$
It follows that
$$
x-\dbE[\F_0(T)\xi]\in\im\cR  \q\Longleftrightarrow\q  x-\dbE[\F_0(T)\xi]\in\im G(0,T),
$$
which proves the result.
\end{proof}

Unlike the Gramian criterion in \cite{yu2021}, which gives a global
characterization of exact controllability, 
\autoref{thm:controllability-characterization} provides a finer, pointwise
characterization of all transferable initial--terminal state pairs. In
particular, it remains informative even when $G(0,T)$ is singular, by identifying
precisely those pairs $(x,\xi)$ that can be connected. Requiring every
initial--terminal state pair to be transferable immediately yields the Gramian
criterion of \cite{yu2021} as a consequence.

\begin{corollary}\label{coro:G(T)-controllability}
System \rf{eq:system} is exactly controllable on $[0,T]$ if and only if
the Gramian $G(0,T)$ defined by \rf{eq:Gramian} is invertible.
\end{corollary}

\subsection{Finite-dimensional and Hautus criteria}

By \autoref{coro:G(T)-controllability}, exact controllability of system
\eqref{eq:system} on $[0,T]$ is reduced to the invertibility of $G(0,T)$.
We now develop finite-dimensional and spectral characterizations of this
condition, beginning with several preparatory observations. Let 
$$
\hA\deq A+\bA,\quad\hC\deq C+\bC,\quad\hB\deq B+\bB,
$$
and define recursively
\begin{align*}
\left\{\begin{aligned}
\cA_0 &=\hB, \q \cB_0=B, \\
\cA_{k+1} &=(\hA\cA_k,\hC\cB_k), \q k\in\dbN, \\
\cB_{k+1} &=(A\cB_k,C\cB_k),     \q k\in\dbN,
\end{aligned}\right.\qq
\left\{\begin{aligned}
V_0 &=\im B, \q U_0=\im\hB, \\
V_{k+1} &=AV_k+CV_k+V_0,       \q k\in\dbN, \\
U_{k+1} &=\hA U_k+\hC V_k+U_0, \q k\in\dbN.
\end{aligned}\right.
\end{align*}
The subspaces $V_k$ defined above are the finite-step reachable subspaces
associated with the corresponding non-mean-field SDE 
\begin{equation}\label{eq:system-SDE}\left\{\begin{aligned}
dX(s) &=[AX(s)+Bu(s)+Cz(s)]ds+z(s)dW(s), \q s\in[t,T],\\
 X(t) &=\eta.
\end{aligned}\right.\end{equation}

The following lemma gives explicit matrix representations of the recursively
defined subspaces $U_k$ and $V_k$. It also records the finite-step
stabilization of $\{V_k\}$, which was established in \cite{sun2026}.

\begin{lemma}\label{lem:U-cA}
For every $k\in\dbN$,
$$
U_k=\im(\cA_0,\cA_1,\ldots,\cA_k), \q V_k=\im(\cB_0,\cB_1,\ldots,\cB_k).
$$
In particular,
$$
U_k\subseteq U_{k+1}, \q V_k\subseteq V_{k+1}, \q k\in\dbN,
$$
and
$$
V_k=V_{n-1}, \q\forall k\ges n-1.
$$
\end{lemma}

\begin{proof}
The first two identities are proved simultaneously by induction. They are
obvious for $k=0$. Suppose that they hold for some $k\in\dbN$. Then
\begin{align*}
U_{k+1} 
&=\hA U_k+\hC V_k+U_0\\
&=\hA\im(\cA_0,\ldots,\cA_k)+\hC\im(\cB_0,\ldots,\cB_k)+\im\cA_0 \\
&=\im(\cA_0,\hA\cA_0,\hC\cB_0,\ldots,\hA\cA_k,\hC\cB_k) \\
&=\im(\cA_0,\cA_1,\ldots,\cA_{k+1}).
\end{align*}
The same argument gives
$$
V_{k+1}=\im(\cB_0,\cB_1,\ldots,\cB_{k+1}).
$$
This completes the induction. The monotonicity follows immediately. The stabilization
$$
V_k=V_{n-1}, \q\forall k\ges n-1,
$$
follows from \cite[Proposition 3.2]{sun2026}.
\end{proof}

\autoref{lem:U-cA} shows that the sequence $\{V_k\}$ stabilizes at $V_{n-1}$. 
We next use this fact to obtain a finite-dimensional representation of $U_{2n-1}$ 
and to show that the sequence $\{U_k\}$ also stabilizes after finitely many steps.

\begin{lemma}\label{lem:U}
We have
$$
U_{2n-1} = \sum_{j=0}^{n-1}\hA^j\big(U_0+\hC V_{n-1}\big).
$$
Moreover,
$$
U_k=U_{2n-1}, \q\forall k\ges2n-1.
$$
\end{lemma}

\begin{proof}
Set
$$
\cK \deq \sum_{j=0}^{n-1}\hA^j\big(U_0+\hC V_{n-1}\big).
$$
By the Cayley--Hamilton theorem, $\hA\cK\subseteq\cK$. Moreover,
$$
U_0\subseteq\cK, \q \hC V_k\subseteq\hC V_{n-1}\subseteq\cK, \q\forall k\in\dbN.
$$
It follows inductively from
$$
U_{k+1}=\hA U_k+\hC V_k+U_0
$$
that $U_k\subseteq\cK$ for every $k\in\dbN$. Hence, $U_{2n-1}\subseteq\cK$. Conversely,
$$
\hA^jU_0\subseteq U_j\subseteq U_{2n-1},
\q
0\les j\les n-1.
$$
Since $\hC V_{n-1}\subseteq U_n$, we also have
$$
\hA^j\hC V_{n-1}\subseteq U_{n+j}\subseteq U_{2n-1}, \q 0\les j\les n-1.
$$
Therefore, $\cK\subseteq U_{2n-1}$, and the first assertion follows.

\ms 

Finally, since
$$
\hA U_{2n-1}\subseteq U_{2n-1}, \q U_0+\hC V_{n-1}\subseteq U_{2n-1},\q
V_k=V_{n-1}, \q\forall k\ges n-1,
$$
it follows that
$$
U_{2n}=\hA U_{2n-1}+\hC V_{2n-1}+U_0\\=\hA U_{2n-1}+\hC V_{n-1}+U_0\subseteq U_{2n-1}.
$$
On the other hand, $U_{2n-1}\subseteq U_{2n}$ by the monotonicity of $\{U_k\}$. Hence,
$$
U_{2n}=U_{2n-1}.
$$
Repeating the same argument inductively yields $U_k=U_{2n-1}$ for all $k\ges2n-1$.
\end{proof}

We now turn to the dual interpretation of $V_{n-1}$.
The following lemma identifies $V_{n-1}^\perp$ as the unobservable 
subspace of the dual dynamics. 

\begin{lemma}\label{lem:zero-condition-SDE}
Let $0\les t<T<\i$, and let $Y(\cd)$ solve
$$
dY(s)=-A^\top Y(s)ds-C^\top Y(s)dW(s), \q Y(t)=y.
$$
Then
$$
B^\top Y(s)=0,\q\as,\q\forall s\in[t,T] \q\Longleftrightarrow\q y\in V_{n-1}^\perp.
$$
\end{lemma}

\begin{proof}
Suppose first that
$$
B^\top Y(s)=0,\q\as,\q \forall s\in[t,T].
$$
Since $\cB_0=B$, assume inductively that
$$
\cB_k^\top Y(s)=0,\q\as,\q \forall s\in[t,T].
$$
Then
$$
0 = d(\cB_k^\top Y) = -(A\cB_k)^\top Y ds-(C\cB_k)^\top Y dW.
$$
By uniqueness of the semimartingale decomposition,
$$
(A\cB_k)^\top Y(s)=0,
\q
(C\cB_k)^\top Y(s)=0,
$$
and hence $\cB_{k+1}^\top Y(s)=0$. By induction,
$$
\cB_k^\top Y(s)=0, \q\as, \q\forall s\in[t,T], \q k\in\dbN.
$$
Evaluating at $s=t$ and using $V_{n-1}=\im(\cB_0,\ldots,\cB_{n-1})$,
we obtain $y\in V_{n-1}^\perp$.

\ms 

Conversely, since $V_k$ stabilizes at $V_{n-1}$,
$$
AV_{n-1}\subseteq V_{n-1}, \q CV_{n-1}\subseteq V_{n-1}.
$$
Thus, $V_{n-1}^\perp$ is invariant under $A^\top$ and $C^\top$. If
$y\in V_{n-1}^\perp$, uniqueness implies
$$
Y(s)\in V_{n-1}^\perp, \q\as, \q\forall s\in[t,T].
$$
Since $\im B\subseteq V_{n-1}$, it follows that
$$
B^\top Y(s)=0, \q\as, \q\forall s\in[t,T].
$$
This completes the proof. 
\end{proof}

We now turn to the dual dynamics of the mean-field system. To extend the
observation criterion in \autoref{lem:zero-condition-SDE} to the mean-field
setting, we separate the dual state into its mean and centered components.

\ms 

For $t\ges0$ and $\beta\in\dbR^n$, let $p(\cd)$ solve
\begin{equation}\label{eq:dual}\left\{\begin{aligned}
dp(s) &=-[A^\top p(s)+\bA^\top\dbE p(s)]ds-[C^\top p(s)+\bC^\top\dbE p(s)]dW(s), \q s\ges t,\\
 p(t) &=\b.
\end{aligned}\right.\end{equation}
Set
$$
q(s)\deq\dbE p(s), \q r(s)\deq p(s)-\dbE p(s).
$$
Then $q(s)$ is deterministic, $\dbE r(s)=0$, and
\begin{equation}\label{eq:q-r}\left\{\begin{aligned}
d\bp q(s)\\r(s)\ep &= -\bm A^\top\bp q(s)\\r(s)\ep ds-\bm C^\top\bp q(s)\\r(s)\ep dW(s),\\
 \bp q(t)\\r(t)\ep &= \bp\b\\0\ep,
\end{aligned}\right.\end{equation}
where
$$
\bm A\deq \begin{pmatrix}\hA&0\\0&A\end{pmatrix}, \q
\bm C\deq \begin{pmatrix}0&\hC\\0&C\end{pmatrix}, \q
\bm B\deq \begin{pmatrix}\hB&0\\0&B\end{pmatrix}.
$$
To apply \autoref{lem:zero-condition-SDE} to the block system \rf{eq:q-r}, define
$$
S_0\deq\im\bm B, \q S_{k+1}\deq\bm A S_k+\bm C S_k+S_0, \q k\in\dbN.
$$
The upper and lower components of vectors in $S_k$ are generally coupled.
Let $\pi_1,\pi_2:\dbR^{2n}\to\dbR^n$ be the canonical projections defined by
$$
\pi_1\bp x\\y\ep=x,
\q
\pi_2\bp x\\y\ep=y.
$$
The following lemma shows that the projections of $S_k$ are precisely $U_k$ and $V_k$.

\begin{lemma}\label{lem:S-projection}
For every $k\in\dbN$,
$$
\pi_1(S_k)=U_k, \q \pi_2(S_k)=V_k.
$$
Consequently, for every $\b\in\dbR^n$,
$$
\bp\b\\0\ep\in S_k^\perp   \q\Longleftrightarrow\q \beta\in U_k^\perp.
$$
\end{lemma}

\begin{proof}
For $k=0$, the conclusion is clear.
If $\pi_1(S_k)=U_k$ and $\pi_2(S_k)=V_k$ for some $k\in\dbN$, 
then by the linearity of $\pi_1,\pi_2$ and the block structures of $\bm A$ and $\bm C$,
\begin{align*}
\pi_1(S_{k+1})
&=\hA\pi_1(S_k)+\hC\pi_2(S_k)+U_0
=U_{k+1},\\
\pi_2(S_{k+1})
&=A\pi_2(S_k)+C\pi_2(S_k)+V_0
=V_{k+1}.
\end{align*}
Thus, the first assertion follows by induction. Finally, 
\begin{align*}
\bp\b\\0\ep\in S_k^\perp 
&\q\Longleftrightarrow\q \lan\bp\b\\0\ep,\z\ran=0, \q\forall\z\in S_k\\
&\q\Longleftrightarrow\q \lan\b,\pi_1(\z)\ran=0,      \q\forall\z\in S_k\\
&\q\Longleftrightarrow\q \b\in\pi_1(S_k)^\perp=U_k^\perp.
\end{align*}
The proof is complete.
\end{proof}

We now turn to the dual interpretation of $U_{2n-1}$.
The following lemma identifies $U_{2n-1}^\perp$ as the unobservable
subspace of the mean-field dual dynamics.

\begin{lemma}\label{lem:zero-condition}
Let $0\les t<T<\i$, and let $p(\cd)$ solve \rf{eq:dual}. Then
$$
B^\top p(s)+\bB^\top\dbE p(s)=0,\q\as,\q\forall s\in[t,T]
\q\Longleftrightarrow\q \b\in U_{2n-1}^\perp.
$$
\end{lemma}

\begin{proof}
Since $p(s)=q(s)+r(s)$ and $\dbE p(s)=q(s)$,
$$
B^\top p(s)+\bB^\top\dbE p(s) = \hB^\top q(s)+B^\top r(s).
$$
As $q(s)$ is deterministic and $\dbE r(s)=0$, the right-hand side vanishes
almost surely if and only if
$$
\hB^\top q(s)=0, \q B^\top r(s)=0,\q\as
$$
Equivalently,
$$
\bm B^\top\bp q(s)\\r(s)\ep=0,\q\as
$$
Applying \autoref{lem:zero-condition-SDE} to the $2n$-dimensional
system \rf{eq:q-r} and then using \autoref{lem:S-projection}, we obtain
$$
B^\top p(s)+\bB^\top\dbE p(s)=0,\q\as,\q\forall s\in[t,T]
\q\Longleftrightarrow\q
\bp\b\\0\ep\in S_{2n-1}^\perp
\q\Longleftrightarrow\q
\b\in U_{2n-1}^\perp.
$$
This completes the proof.  
\end{proof}

\begin{remark}
Combining \autoref{lem:U-cA} and \autoref{lem:U}, we obtain
$$
U_{2n-1}=\im(\cA_0,\ldots,\cA_{2n-1})=\im(\cA_0,\cA_1,\ldots).
$$
Hence, \autoref{lem:zero-condition} yields
$$
B^\top p(s)+\bB^\top\dbE p(s)=0,~\as,~\forall s\in[t,T]
~\Leftrightarrow~
\cA_j^\top\beta=0,~0\les j\les2n-1
~\Leftrightarrow~
\cA_j^\top\beta=0,~j\in\dbN.
$$
Thus, the infinite family of orthogonality conditions appearing in
\cite{yu2021} reduces to finitely many algebraic conditions.
\end{remark}

Although the present section concerns the zero initial time, the Gramian argument 
is naturally valid on an arbitrary interval $[t,T]$. For $0\les t<T<\i$, define
\begin{equation}\label{eq:G(t,T)}
G(t,T)\deq\dbE\int_t^T\(\F_t(s)B+[\dbE\F_t(s)]\bB\)\(\F_t(s)B+[\dbE\F_t(s)]\bB\)^\top ds,
\end{equation}
where $\F_t(\cd)$ solves
\begin{equation}\label{eq:Phi-t}\left\{\begin{aligned}
d\F_t(s) &=-\big\{\F_t(s)A+[\dbE\F_t(s)]\bA\big\}ds
           -\big\{\F_t(s)C+[\dbE\F_t(s)]\bC\big\}dW(s), \q s\ges t,\\
 \F_t(t) &=I_n.
\end{aligned}\right.
\end{equation}
Taking expectations in \eqref{eq:Phi-t} gives
$$
\dbE\F_t(s)=e^{-(s-t)\hA}, \q s\in[t,T].
$$

The following theorem characterizes the invertibility of $G(t,T)$ by a finite-dimensional 
condition that is independent of the particular interval.

\begin{theorem}\label{thm:G(t,T)}
Let $0\les t<T<\i$. Then the following statements are equivalent:
\begin{enumerate}[\rm(i)]
\item $G(t,T)$ is invertible.
\item $U_{2n-1}=\dbR^n$.
\item $\rank(\cA_0,\cA_1,\ldots,\cA_{2n-1})=n$.
\end{enumerate}
\end{theorem}

\begin{proof}
By \autoref{lem:U-cA}, statements (ii) and (iii) are equivalent. For $\b\in\dbR^n$, set
$$
p(s)=\F_t(s)^\top\b, \q s\ges t.
$$
Then $p(\cd)$ solves \rf{eq:dual}, and
$$
\b^\top G(t,T)\b = \dbE\int_t^T\big|B^\top p(s)+\bB^\top\dbE p(s)\big|^2ds.
$$
By the continuity of $p(\cd)$ and \autoref{lem:zero-condition},
$$
\b^\top G(t,T)\b=0 \q\Longleftrightarrow\q
B^\top p(s)+\bB^\top\dbE p(s)=0, \q\as,~\forall s\in[t,T]
\q\Longleftrightarrow\q
\b\in U_{2n-1}^\perp.
$$
Since $G(t,T)$ is positive semidefinite, this yields $\ker G(t,T)=U_{2n-1}^\perp$. Therefore,
$$
G(t,T)\text{ is invertible}
\q\Longleftrightarrow\q
U_{2n-1}^\perp=\{0\}
\q\Longleftrightarrow\q
U_{2n-1}=\dbR^n.
$$
This completes the proof. 
\end{proof}

When $t=0$, equation \eqref{eq:Phi-t} coincides with \rf{eq:Phi}. 
Hence, the Gramian defined by \rf{eq:G(t,T)} coincides with $G(0,T)$ defined by \rf{eq:Gramian}.
Combining \autoref{coro:G(T)-controllability} with \autoref{thm:G(t,T)}, we obtain the following 
zero-time Hautus criterion.

\begin{theorem}\label{thm:criterion-U}
Let $T>0$, and let
$$
P=(v_1,\ldots,v_r)\in\dbR^{n\times r},  \q r=\dim V_{n-1},
$$
where $v_1,\ldots,v_r$ form a basis of $V_{n-1}$. Then the following statements are equivalent:
\begin{enumerate}[\rm(i)]
\item System \rf{eq:system} is exactly controllable on $[0,T]$.
\item $U_{2n-1}=\dbR^n$.
\item The deterministic pair $[\hA,(\hB,\hC P)]$ is controllable.
\item For every $\l\in\dbC$ and $x\in\dbC^n$,
      $$
      \hA^{\,\top} x=\l x, \q\hB^\top x=0, \q P^\top\hC^\top x=0 \q\Longrightarrow\q x=0.
      $$
\item For every $\l\in\sigma(\hA)$, $\rank(\l I_n-\hA,\hB,\hC P)=n$.
\end{enumerate}
\end{theorem}

\begin{proof}
By \autoref{coro:G(T)-controllability} and \autoref{thm:G(t,T)} with $t=0$, 
(i) and (ii) are equivalent. Since
$$
U_0=\im\hB, \q \im P=V_{n-1},
$$
we have
$$
U_0+\hC V_{n-1}=\im(\hB,\hC P).
$$
Hence, \autoref{lem:U} gives
$$
U_{2n-1} = \sum_{j=0}^{n-1}\hA^j\big(U_0+\hC V_{n-1}\big)
         = \im\Big((\hB,\hC P),\hA(\hB,\hC P),\ldots,\hA^{n-1}(\hB,\hC P)\Big).
$$
Thus, statement (ii) is equivalent to statement (iii).
The equivalence of (iii)--(v) follows from the classical Kalman--Hautus criterion. 
\end{proof}

Since $\im(\hC P)=\hC V_{n-1}$, the criterion does not depend on the particular 
basis chosen for $V_{n-1}$. 
More importantly, at the zero initial time, the centered stochastic dynamics can 
help control the mean dynamics through the additional directions $\hC V_{n-1}$. 
Therefore, even if the mean pair $[\hA,\hB]$ is not controllable by itself, 
the full mean-field system \rf{eq:system} may still be exactly controllable.

\ms 

As an immediate consequence, controllability of the mean equation alone
is sufficient for exact controllability of the mean-field system.

\begin{corollary}\label{coro:hA-hB}
Suppose that the deterministic pair $[\hA,\hB]$ is controllable. Then
system \rf{eq:system} is exactly controllable on $[0,T]$.
\end{corollary}

\begin{proof}
Let $P$ be as in \autoref{thm:criterion-U}. 
Since $[\hA,\hB]$ is controllable, so is $[\hA,(\hB,\hC P)]$. 
The conclusion follows from \autoref{thm:criterion-U}.
\end{proof}

\section{Exact Controllability at Positive Initial Times}\label{sec:controllability-t>0}

\subsection{Operator and Gramian characterizations}

We now consider exact controllability of system \eqref{eq:system} when the
initial time is positive. Unlike the case $t=0$, an $\sF_t$-measurable
initial state may contain a nontrivial centered random component. 
We first derive operator and Gramian characterizations that separate the mean 
and centered components.

\ms

We begin with two equivalent formulations in which either the initial or the
terminal state is fixed.

\begin{proposition}\label{prop:equivalence-controllability}
Let $0\les t<T<\i$. Then the following statements are equivalent:
\begin{enumerate}[\rm(i)]
\item System \eqref{eq:system} is exactly controllable on $[t,T]$.
\item For every $\xi\in L_{\sF_T}^2(\Om;\dbR^n)$, there exists
      $(u(\cd),z(\cd))\in\dbU[t,T]$ such that
      $$
      X(T;t,0,u,z)=\xi,\q\as
      $$
\item For every $\eta\in L_{\sF_t}^2(\Om;\dbR^n)$, there exists
      $(u(\cd),z(\cd))\in\dbU[t,T]$ such that
      $$
      X(T;t,\eta,u,z)=0,\q\as
      $$
\end{enumerate}
\end{proposition}

\begin{proof}
The implications {\rm(i)}$\Rightarrow${\rm(ii)} and
{\rm(i)}$\Rightarrow${\rm(iii)} are immediate.

\ms 

Assume (ii) and let $\eta\in L_{\sF_t}^2(\Om;\dbR^n)$ and
$\xi\in L_{\sF_T}^2(\Om;\dbR^n)$. Set
$$
X_0(\cd)\deq X(\cd\,;t,\eta,0,0).
$$
By (ii), there exists $(u(\cd),z(\cd))\in\dbU[t,T]$ such that
$$
X(T;t,0,u,z)=\xi-X_0(T),\q\as
$$
By linearity,
$$
X(\cd\,;t,\eta,u,z)=X_0(\cd)+X(\cd\,;t,0,u,z),
$$
and hence
$$
X(T;t,\eta,u,z)=\xi,\q\as
$$
Thus (i) follows.

\ms 

Assume (iii). Let $(X_0(\cd),z_0(\cd))$ be the adapted solution of
$$
\left\{\begin{aligned}
dX_0(s) &= [AX_0(s)+\bA\dbE X_0(s)+Cz_0(s)+\bC\dbE z_0(s)]ds+z_0(s)dW(s),\\
 X_0(T) &= \xi.
\end{aligned}\right.
$$
By (iii), there exists a control pair $(u(\cd),z(\cd))\in\dbU[t,T]$ such that
$$
X(T;t,\eta-X_0(t),u,z)=0,\q\as
$$
By linearity,
$$
X(\cd\,;t,\eta,u,z+z_0)=X_0(\cd)+X(\cd\,;t,\eta-X_0(t),u,z),
$$
and therefore
$$
X(T;t,\eta,u,z+z_0)=\xi,\q\as
$$
Hence (i) follows. 
\end{proof}

\begin{remark}
Statement (iii) is the exact null-controllability of
system \eqref{eq:system}; see \cite{yu2021} for the mean-field setting
and \cite{wang2017} for related linear stochastic systems. 
The equivalence of (i) and (iii) was previously established
for $t=0$ in \cite{yu2021}. 
\autoref{prop:equivalence-controllability} extends this result to every
$t\ges0$ and further proves the equivalence of {\rm(i)} and {\rm(ii)}.
Thus, exact controllability may be tested either by steering an arbitrary
initial state to zero or by considering reachability from the zero initial
state.
\end{remark}

Fix $0\les t<T<\i$ and define the bounded linear operator
$\cN_t:\dbU[t,T]\longrightarrow L_{\sF_T}^2(\Om;\dbR^n)$ by
\begin{equation}\label{eq:Nt}
\cN_t(u,z)\deq X(T;t,0,u,z).
\end{equation}
By \autoref{prop:equivalence-controllability}, system \eqref{eq:system}
is exactly controllable on $[t,T]$ if and only if $\cN_t$ is onto. For
bounded operators between Hilbert spaces, this is equivalent to $\cN_t^*$
being bounded below, that is, there exists $\d>0$ such that
$$
\|\cN_t^*\xi\|_{\dbU[t,T]}\ges\d\|\xi\|_{L_{\sF_T}^2},\q\forall\xi\in L_{\sF_T}^2(\Om;\dbR^n).
$$
The following lemma gives an explicit representation of the adjoint
operator $\cN_t^*$ in terms of the associated mean-field backward equation.

\begin{lemma}\label{lem:Nt*}
For $\xi\in L_{\sF_T}^2(\Om;\dbR^n)$, let $(Y(\cd),Z(\cd))$ be the adapted solution of
\begin{equation}\label{eq:dual-system-initial-0-MF}
\left\{\begin{aligned}
dY(s) &= -[A^\top Y(s)+\bA^\top\dbE Y(s)]ds+Z(s)dW(s),\\
 Y(T) &= \xi.
\end{aligned}\right.
\end{equation}
Then
$$
\cN_t^*\xi=\big(B^\top Y+\bB^\top\dbE Y,\,C^\top Y+\bC^\top\dbE Y+Z\big).
$$
\end{lemma}

\begin{proof}
Let $X(\cd)=X(\cd\,;t,0,u,z)$. It\^o's formula gives
\begin{align*}
\dbE\lan X(T),\xi\ran=\dbE\int_t^T\Big[\lan B^\top Y+\bB^\top\dbE Y,u\ran
+\lan C^\top Y+\bC^\top\dbE Y+Z,z\ran\Big]ds,
\end{align*}
which yields the asserted formula.
\end{proof}

For the centered dynamics, let $\varPsi_t(\cd)$ solve 
\begin{equation}\label{eq:Psi-t}
\left\{\begin{aligned}
d\varPsi_t(s) &= -\varPsi_t(s)A\,ds-\varPsi_t(s)C\,dW(s), \q s\ges t,\\
 \varPsi_t(t) &= I_n,
\end{aligned}\right.
\end{equation}
and define
\begin{equation}\label{eq:bG(t,T)}
\bar G(t,T)\deq\dbE\int_t^T\varPsi_t(s)BB^\top\varPsi_t(s)^\top ds.
\end{equation}
The next lemma separates the mean and centered parts of the homogeneous dual dynamics.

\begin{lemma}\label{lem:representation-MF-SDE}
Let $\eta\in L_{\sF_t}^2(\Om;\dbR^n)$, and let $Y(\cd)$ be the solution of
\begin{equation}\label{eq:homogeneous-dual}
\left\{\begin{aligned}
dY(s) &=-[A^\top Y(s)+\bA^\top\dbE Y(s)]ds
        -[C^\top Y(s)+\bC^\top\dbE Y(s)]dW(s), \q s\in[t,T],\\
 Y(t) &=\eta.
\end{aligned}\right.
\end{equation}
Then
$$
Y(s)=\F_t(s)^\top\dbE\eta+\varPsi_t(s)^\top(\eta-\dbE\eta),\q s\in[t,T],
$$
where $\F_t(\cd)$ is the solutio of \rf{eq:Phi-t}. 
\end{lemma}

\begin{proof}
Set
$$
m\deq\dbE\eta, \q \eta_0\deq\eta-\dbE\eta,
$$
and define
$$
\widetilde Y(s)\deq\F_t(s)^\top m+\varPsi_t(s)^\top\eta_0.
$$
Since $\F_t(\cd)$ and $\varPsi_t(\cd)$ depend only on the Brownian
increments after time $t$, they are independent of $\sF_t$. Hence,
using $\dbE\eta_0=0$,
$$
\dbE[\varPsi_t(s)^\top\eta_0]=0,
$$
and therefore
$$
\dbE\widetilde Y(s)=[\dbE\F_t(s)]^\top m.
$$
Taking transposes in the equations for $\F_t$ and $\varPsi_t$ gives
\begin{align*}
d\F_t^\top 
&=-[A^\top\F_t^\top+\bA^\top(\dbE\F_t)^\top]ds-[C^\top\F_t^\top+\bC^\top(\dbE\F_t)^\top]dW,\\
d\varPsi_t^\top 
&=-A^\top\varPsi_t^\top ds-C^\top\varPsi_t^\top dW.
\end{align*}
Consequently,
\begin{align*}
d\widetilde Y
&=-\big[A^\top\big(\F_t^\top m+\varPsi_t^\top\eta_0\big)+\bA^\top(\dbE\F_t)^\top m\big]ds
  -\big[C^\top\big(\F_t^\top m+\varPsi_t^\top\eta_0\big)+\bC^\top(\dbE\F_t)^\top m\big]dW \\
&=-\big[A^\top\widetilde Y+\bA^\top\dbE\widetilde Y\big]ds
  -\big[C^\top\widetilde Y+\bC^\top\dbE\widetilde Y\big]dW. 
\end{align*}
Moreover, $\widetilde Y(t)=m+\eta_0=\eta$.
Thus $\widetilde Y$ solves \eqref{eq:homogeneous-dual} with the same
initial condition as $Y$. By uniqueness, $Y(s)=\widetilde Y(s)$.
\end{proof}

We are now in a position to give the main result of this subsection, 
which connects exact controllability at a positive initial time with 
an observability estimate for $\cN_t^*$ and with the invertibility of 
the Gramians corresponding to the mean and centered dynamics.

\begin{theorem}\label{thm:MF-controllability-t-T}
Let $0<t<T<\i$. Then the following statements are equivalent:
\begin{enumerate}[\rm(i)]
\item System \eqref{eq:system} is exactly controllable on $[t,T]$.
\item There exists $\d>0$ such that
      $$
      \|\cN_t^*\xi\|_{\dbU[t,T]}\ges\d\|\xi\|_{L_{\sF_T}^2},
      \q\forall\xi\in L_{\sF_T}^2(\Om;\dbR^n).
      $$
\item Both $G(t,T)$ and $\bar G(t,T)$ are invertible. 
\end{enumerate}
\end{theorem}

\begin{proof}
The equivalence of (i) and (ii) follows from
\autoref{prop:equivalence-controllability} and the standard surjectivity
criterion for bounded linear operators.

\ms

We prove (ii)$\Rightarrow$(iii). If $G(t,T)$ is singular,
choose $0\ne v\in\ker G(t,T)$ and set 
$$
Y(s)\deq \F_t(s)^\top v,\q \xi\deq Y(T), \q Z(s)\deq -[C^\top Y(s)+\bC^\top\dbE Y(s)].
$$
Then $(Y(\cd),Z(\cd))$ solves \eqref{eq:dual-system-initial-0-MF}, and \autoref{lem:Nt*} gives
$$
\|\cN_t^*\xi\|_{\dbU[t,T]}^2=v^\top G(t,T)v=0.
$$
By (ii), $\xi=0$. Taking expectations yields $e^{-(T-t)\hA^\top}v=0$, contradicting $v\ne0$. 

\ms 

If $\bar G(t,T)$ is singular, choose $0\ne v\in\ker\bar G(t,T)$ and set
$$
\th\deq\frac{W(t)}{\sqrt t}, \q Y(s)\deq\varPsi_t(s)^\top\th v.
$$
Then $\dbE\th=0$, $\dbE\th^2=1$, and the independence of
$\varPsi_t(\cd)$ from $\sF_t$ gives $\dbE Y(s)=0$. With
$$
\xi\deq Y(T),\q Z(s)\deq-C^\top Y(s),
$$
we have by \autoref{lem:Nt*},
\begin{align*}
\|\cN_t^*\xi\|_{\dbU[t,T]}^2
&=\dbE\int_t^T|B^\top Y(s)|^2ds=\dbE\int_t^T\th^2|B^\top\varPsi_t(s)^\top v|^2ds\\
&=\dbE\th^2\,\dbE\int_t^T|B^\top\varPsi_t(s)^\top v|^2ds =v^\top\bar G(t,T)v=0,
\end{align*}
where we used the independence of $\th$ and $\varPsi_t(\cd)$ and
$\dbE\th^2=1$. By (ii), it follows that $\xi=0$. On the other hand,
$$
0=\xi=Y(T)=\th\,\varPsi_t(T)^\top v.
$$
Since $\th\ne0$ almost surely and $\varPsi_t(T)$ is invertible almost surely, 
the above implies $v=0$, a contradiction.  

\ms 

It remains to prove (iii)$\Rightarrow$(ii). 
Let $(Y(\cd),Z(\cd))$ be the adapted solution to \eqref{eq:dual-system-initial-0-MF} and set
$$
\cO(s)\deq B^\top Y(s)+\bB^\top\dbE Y(s),
\q
\D(s)\deq C^\top Y(s)+\bC^\top\dbE Y(s)+Z(s).
$$
Let $\widetilde Y(\cd)$ solve \eqref{eq:homogeneous-dual} with
$\widetilde Y(t)=Y(t)$ and define
$$
\widetilde\cO(s) \deq B^\top\widetilde Y(s)+\bB^\top\dbE\widetilde Y(s), \q 
m\deq\dbE Y(t), \q \eta_0\deq Y(t)-\dbE Y(t).
$$
By \autoref{lem:representation-MF-SDE},
$$
\widetilde Y(s)=\F_t(s)^\top m+\varPsi_t(s)^\top\eta_0.
$$
Since $\eta_0$ is $\sF_t$-measurable with $\dbE\eta_0=0$, whereas
$\F_t(\cd)$ and $\varPsi_t(\cd)$ depend only on the Brownian increments
after time $t$, we have
$$
\dbE\widetilde Y(s)=[\dbE\F_t(s)]^\top m.
$$
Hence, 
\begin{align*}
\widetilde\cO(s)
=B^\top\widetilde Y(s)+\bB^\top\dbE\widetilde Y(s)
=\big\{\F_t(s)B+[\dbE\F_t(s)]\bB\big\}^\top m+B^\top\varPsi_t(s)^\top\eta_0.
\end{align*}
The cross term vanishes after taking expectations, since $\eta_0$ has
zero mean and is independent of the future Brownian increments. Therefore,
by the definitions of $G(t,T)$ and $\bar G(t,T)$,
\begin{align*}
\dbE\int_t^T|\widetilde\cO(s)|^2ds
&=m^\top G(t,T)m + \dbE\big[\eta_0^\top\bar G(t,T)\eta_0\big] \\
&=[\dbE Y(t)]^\top G(t,T)\dbE Y(t)+
\dbE\Big\{[Y(t)-\dbE Y(t)]^\top\bar G(t,T)[Y(t)-\dbE Y(t)]\Big\}.
\end{align*}
Since both Gramians are positive definite, there exists $c>0$ such that
$$
c\,\dbE|Y(t)|^2\les\dbE\int_t^T|\widetilde\cO(s)|^2ds.
$$
Set $R=Y-\widetilde Y$. Then
$$
dR=-[A^\top R+\bA^\top\dbE R]ds-[C^\top R+\bC^\top\dbE R]dW+\D\,dW, \q R(t)=0.
$$
Standard estimates give
$$
\dbE\sup_{t\les s\les T}|R(s)|^2\les K\dbE\int_t^T|\D(s)|^2ds.
$$
Since
$$
\widetilde\cO=\cO-[B^\top R+\bB^\top\dbE R],
$$
we obtain
$$
\dbE|Y(t)|^2\les K\dbE\int_t^T\big(|\cO(s)|^2+|\D(s)|^2\big)ds.
$$
Finally, since
$$
Z=\D-C^\top Y-\bC^\top\dbE Y,
$$
the equation for $Y$ can be written as
$$
dY=-[A^\top Y+\bA^\top\dbE Y]ds+[\D-C^\top Y-\bC^\top\dbE Y]dW.
$$
The standard $L^2$ estimate yields
$$
\dbE|\xi|^2
\les K\lt(\dbE|Y(t)|^2+\dbE\int_t^T|\D(s)|^2ds\rt)
\les K\|\cN_t^*\xi\|_{\dbU[t,T]}^2.
$$
Thus $\cN_t^*$ is bounded below. 
\end{proof}

\subsection{Interval independence and the positive-time Hautus criterion}

Because the coefficients are constant, the fundamental matrices are invariant
in law under time shifts. To see this, for $r\ges0$, set
$$
W_t(r)\deq W(t+r)-W(t), \q
\widehat\F_t(r)\deq\F_t(t+r), \q
\widehat\varPsi_t(r)\deq\varPsi_t(t+r).
$$
Then $W_t(\cd)$ is a standard Brownian motion. 
By \eqref{eq:Phi-t}, $\widehat\F_t(\cd)$ satisfies
$$
\left\{\begin{aligned}
d\widehat\F_t(r) &= -\big\{\widehat\F_t(r)A+[\dbE\widehat\F_t(r)]\bA\big\}dr 
                    -\big\{\widehat\F_t(r)C+[\dbE\widehat\F_t(r)]\bC\big\}dW_t(r),\\
 \widehat\F_t(0) &= I_n.
\end{aligned}\right.
$$
Thus, $\widehat\F_t(\cd)$ satisfies the same equation as $\F_0(\cd)$,
with $W_t(\cd)$ in place of $W(\cd)$. Since $W_t(\cd)$ is again a
standard Brownian motion, uniqueness in law gives
$$
\{\F_t(t+r):r\ges0\} \stackrel{d}{=} \{\F_0(r):r\ges0\}.
$$
Similarly, we have
$$
\{\varPsi_t(t+r):r\ges0\} \stackrel{d}{=} \{\varPsi_0(r):r\ges0\}.
$$
Consequently,
\begin{equation}\label{eq:shifted-Gramians}
G(t,T)=G(0,T-t),
\q
\bar G(t,T)=\bar G(0,T-t).
\end{equation}
Thus, the Gramians depend on the interval only through its length. Together
with the finite-dimensional criteria, this will imply that their invertibility
is independent of every positive control horizon.

\ms 

For the non-mean-field SDE \eqref{eq:system-SDE}, the equivalence of the
Gramian, finite-dimensional, and Hautus criteria at the zero initial time
was established in \cite{sun2026}. The next proposition shows that the same
characterizations remain valid for every positive initial time. Consequently,
exact controllability of \eqref{eq:system-SDE} is independent of the initial
time and, in fact, of the particular control interval.

\begin{proposition}\label{prop:SDE-controllability}
Let $0\les t<T<\i$. Then the following statements are equivalent:
\begin{enumerate}[\rm(i)]
\item System \eqref{eq:system-SDE} is exactly controllable on $[t,T]$.
\item $\bar G(t,T)$ is invertible.
\item $V_{n-1}=\dbR^n$.
\item $B^\top H\ne0$ for every nonzero positive-semidefinite eigenmatrix $H$ of
      $$
      \cL_{(-A,C)}(M)\deq-MA-A^\top M+C^\top MC,\q M\in\dbS^n.
      $$
\end{enumerate}
\end{proposition}

\begin{proof}
For $t=0$, the equivalence of (i)--(iv) was established in \cite{sun2026}. 
It therefore remains to consider the case $t>0$. Setting
$$
\bA=\bB=\bC=0
$$
in \autoref{thm:MF-controllability-t-T}, we have
$$
\F_t=\varPsi_t, \q G(t,T)=\bar G(t,T).
$$
Hence, (i)$\Leftrightarrow$(ii). Moreover, when $\bA=\bB=\bC=0$,
the recursions defining $U_k$ and $V_k$ coincide. Thus, by \autoref{thm:G(t,T)},
(ii)$\Leftrightarrow$(iii). 
Finally, the equivalence of (iii) and (iv) is precisely the
stochastic Hautus criterion established in \cite{sun2026}.
\end{proof}

\begin{corollary}\label{coro:MF-SDE-t>0-SDE}
Let $0<t<T<\i$. If system \eqref{eq:system} is exactly controllable on
$[t,T]$, then system \eqref{eq:system-SDE} is exactly controllable on every
interval $[t_1,t_2]$ with $0\les t_1<t_2<\i$.
\end{corollary}

\begin{proof}
By \autoref{thm:MF-controllability-t-T}, $\bar G(t,T)$ is invertible.
The conclusion follows from \autoref{prop:SDE-controllability}.
\end{proof}

We are now ready to state the main result of this section. 
It gives a complete finite-dimensional and spectral characterization of 
exact controllability at positive initial times. 
In particular, it shows that controllability on one positive-time interval 
is equivalent to controllability on every such interval, 
and that the obstruction consists of two independent parts: 
a stochastic Hautus condition for the centered dynamics and 
a classical Hautus condition for the mean dynamics.

\begin{theorem}\label{thm:controllability-t>0}
The following statements are equivalent:
\begin{enumerate}[\rm(i)]
\item System \eqref{eq:system} is exactly controllable on every interval
      $[t,T]$ with $0<t<T<\i$. 
\item System \eqref{eq:system} is exactly controllable on some interval
      $[t,T]$ with $0<t<T<\i$.
\item $U_{2n-1}=V_{n-1}=\dbR^n$. 
\item $V_{n-1}=\dbR^n$ and the deterministic pair $[\hA,(\hB,\hC)]$ is controllable.
\item Both of the following spectral conditions hold:
      \begin{enumerate}[\rm(a)]
      \item $B^\top H\ne0$ for every nonzero positive-semidefinite
            eigenmatrix $H$ of $\cL_{(-A,C)}$;
      \item for every $\l\in\dbC$ and $x\in\dbC^n$,
            $$
            \hA^\top x=\l x, \q\hB^\top x=0, \q\hC^\top x=0 \q\Longrightarrow\q x=0,
            $$
            or equivalently, 
            $$
            \rank(\l I_n-\hA,\hB,\hC)=n,\q\forall\l\in\si(\hA).
            $$
      \end{enumerate}
\end{enumerate}
\end{theorem}

\begin{proof}
We first show that (i)--(iii) are equivalent.
Let $0<t<T<\i$. By \autoref{thm:MF-controllability-t-T},
system \eqref{eq:system} is exactly controllable on $[t,T]$ if and only if
both $G(t,T)$ and $\bar G(t,T)$ are invertible. By \autoref{thm:G(t,T)},
$$
G(t,T)\text{ is invertible}\q\Longleftrightarrow\q U_{2n-1}=\dbR^n,
$$
while, by \autoref{prop:SDE-controllability},
$$
\bar G(t,T)\text{ is invertible}\q\Longleftrightarrow\q V_{n-1}=\dbR^n.
$$
Therefore, system \eqref{eq:system} is exactly controllable on $[t,T]$
if and only if $U_{2n-1}=V_{n-1}=\dbR^n$.
Since the latter condition does not depend on $t$ or $T$, statements
(i)--(iii) are equivalent.

\ms 

We next prove (iii)$\Leftrightarrow$(iv).
Under the condition $V_{n-1}=\dbR^n$, we may choose $P=I_n$ in
\autoref{thm:criterion-U}. Hence,
$$
U_{2n-1}=\dbR^n \q\Longleftrightarrow\q [\hA,(\hB,\hC)]\text{ is controllable}.
$$
Thus, (iii) and (iv) are equivalent.

\ms 

Finally, by \autoref{prop:SDE-controllability},
$V_{n-1}=\dbR^n$ is equivalent to part (a) of (v).
On the other hand, by the classical Hautus criterion,
$[\hA,(\hB,\hC)]$ is controllable if and only if, for every $\l\in\dbC$ and $x\in\dbC^n$,
$$
\hA^\top x=\l x, \q\hB^\top x=0, \q\hC^\top x=0 \q\Longrightarrow\q x=0,
$$
which is part (b) of (v). 
Hence (iv) and (v) are equivalent, and the proof is complete. 
\end{proof}

Combining \autoref{thm:criterion-U} with \autoref{thm:controllability-t>0} 
makes the initial-time dichotomy explicit.
At $t=0$, exact controllability is equivalent to
$$
U_{2n-1}=\dbR^n,
$$
whereas for every $t>0$ it is equivalent to the stronger condition
$$
U_{2n-1}=V_{n-1}=\dbR^n.
$$
Thus, positive-time exact controllability requires, in addition, exact
controllability of the centered stochastic system. The distinction comes
from the initial sigma-field: $\sF_0$ is trivial, so the initial state at
time zero is necessarily deterministic, while for $t>0$ an
$\sF_t$-measurable initial state may contain an arbitrary centered random
component.

\section{Relationships among the Associated Systems}\label{sec:relationship}

We conclude by comparing exact controllability of the MF-SDE
\eqref{eq:system}, the corresponding non-mean-field SDE
\eqref{eq:system-SDE}, and the deterministic system
\begin{equation}\label{eq:system-ODE}
\left\{\begin{aligned}
\dot x(s) &= Ax(s)+Bu(s), \q s\in[t,T],\\
     x(t) &= x_0.
\end{aligned}\right.
\end{equation}
For \eqref{eq:system-ODE}, controllability is understood in the classical Kalman sense.

\ms 

The following theorem summarizes the exact controllability relations among these 
three systems and also clarifies the role of the initial time for the MF-SDE.

\begin{theorem}\label{thm:relation}
Consider the following statements:
\begin{enumerate}[\rm(i)]
\item The deterministic pair $[A,B]$ is controllable.
\item System \eqref{eq:system-SDE} is exactly controllable on some
      (equivalently, every) interval $[t,T]$ with $0\les t<T<\i$.
\item System \eqref{eq:system} is exactly controllable on some
      (equivalently, every) interval $[t,T]$ with $0<t<T<\i$.
\item System \eqref{eq:system} is exactly controllable on $[0,T]$ for some
      (equivalently, every) $T>0$.
\item The deterministic pair $[\hA,(\hB,\hC)]$ is controllable.
\end{enumerate}
The following implications and equivalences hold:
$$
{\rm(i)}\Longrightarrow{\rm(ii)},\q
{\rm(iii)} \Longleftrightarrow {\rm(ii)}+{\rm(iv)} \Longleftrightarrow {\rm(ii)}+{\rm(v)}.
$$
\end{theorem}

\begin{proof}
Since $\im(B,AB,\ldots,A^{n-1}B)\subseteq V_{n-1}$,
controllability of $[A,B]$ implies $V_{n-1}=\dbR^n$. 
Hence (i) implies (ii) by \autoref{prop:SDE-controllability}. 

\ms 

Next, by \autoref{thm:controllability-t>0},
$$
{\rm(iii)} \q\Longleftrightarrow\q U_{2n-1}=V_{n-1}=\dbR^n.
$$
On the other hand, \autoref{thm:criterion-U} gives
$$
{\rm(iv)} \q\Longleftrightarrow\q U_{2n-1}=\dbR^n,
$$
while \autoref{prop:SDE-controllability} gives
$$
{\rm(ii)} \q\Longleftrightarrow\q  V_{n-1}=\dbR^n.
$$
Therefore, (iii)$\Leftrightarrow$(ii)+(iv).

\ms 

Finally, under (ii), we have $V_{n-1}=\dbR^n$, so that
$P=I_n$ may be chosen in \autoref{thm:criterion-U}. Thus,
$$
{\rm(iv)} \q\Longleftrightarrow\q [\hA,(\hB,\hC)] \text{ is controllable},
$$
which is precisely statement {\rm(v)}. 
Hence, ${\rm(ii)}+{\rm(iv)}\Leftrightarrow{\rm(ii)}+{\rm(v)}$. 
\end{proof}

Thus, exact controllability of the MF-SDE at a positive initial time is
equivalent to the simultaneous exact controllability of the zero-time
MF-SDE and the corresponding non-mean-field SDE. This relation gives
another interpretation of the initial-time dichotomy established in
\autoref{thm:controllability-t>0}.

\ms 

The implication (i)$\Rightarrow$(ii) in \autoref{thm:relation} is, in general, strict, as demonstrated by the following example.

\begin{example}\label{ex:SDE-not-ODE}
Let
$$
A=0, \q B=\bp1\\0\ep, \q C=\bp0&1\\1&0\ep.
$$
Then
$$
V_1=\im B+C\im B=\dbR^2.
$$
Therefore, by \autoref{prop:SDE-controllability}, the SDE
\eqref{eq:system-SDE} is exactly controllable. On the other hand,
$$
\rank(B,AB)=1,
$$
so the deterministic pair $(A,B)$ is not controllable. Thus, the
converse of the implication in \autoref{thm:relation} does not hold in
general.
\end{example}

At the zero initial time, exact controllability of the non-mean-field 
SDE \rf{eq:system-SDE} and that of the MF-SDE \rf{eq:system} are
independent in general, as the following example shows.

\begin{example}\label{ex:zero-time-independence}
First, consider 
\begin{equation}\label{eq:coefficient-example}
\left\{\begin{aligned}
A&=0, &  B&=\bp1&0\\0&0\ep, & C&=0, \\
\bA&=I_2, & \bB&=\bp0&0\\0&1\ep, & \bC&=I_2.
\end{aligned}\right.
\end{equation}
Then
$$
V_1=\operatorname{span}\{e_1\}\ne\dbR^2, \q \hB=I_2.
$$
Thus, by \autoref{prop:SDE-controllability}, the non-mean-field SDE \rf{eq:system-SDE} 
is not exactly controllable. On the other hand,
$$
U_0=\im\hB=\dbR^2,
$$
and hence $U_{2n-1}=\dbR^2$. Therefore, by \autoref{thm:criterion-U}, 
the MF-SDE \rf{eq:system} is exactly controllable at the zero initial time. 
However, since $V_1\ne\dbR^2$, \autoref{thm:controllability-t>0} shows that 
it is not exactly controllable on any interval with positive initial time. 
Thus, this example also illustrates the initial-time dichotomy.

\ms 

Conversely, consider
$$
A=0, \q B=\bp1&0\\0&0\ep, \q C=\bp0&1\\1&0\ep, \q \bA=I_2, \q \bB=-B, \q \bC=-C.
$$
We have 
$$
V_1=\im B+C\im B=\dbR^2, \q \hB=\hC=0.
$$
Hence, by \autoref{prop:SDE-controllability}, the corresponding non-mean-field SDE 
\rf{eq:system-SDE} is exactly controllable. 
In contrast, $U_0=\{0\}$, and the recursion for $U_k$, together with $\hC=0$, gives
$$
U_k=\{0\}, \q\forall k\in\dbN.
$$
Therefore, by \autoref{thm:criterion-U}, the MF-SDE \rf{eq:system} is not exactly 
controllable at the zero initial time. 
By \autoref{thm:controllability-t>0}, it is not exactly
controllable at any positive initial time either.
\end{example}

For convenience, Figure \ref{fig:summary} summarizes the main controllability
criteria established above and the relationships among the associated systems.
The two upper blocks give simple sufficient conditions. The two middle blocks
describe exact controllability of the MF-SDE at the zero initial time and of
the corresponding non-mean-field SDE, respectively, while the lower block
gives the characterization at positive initial times. Here $P$ denotes any
matrix whose columns form a basis of $V_{n-1}$, and PSD stands for
positive semidefinite. A single arrow in the figure denotes a sufficient
implication. In particular, the two arrows from the lower block indicate that
exact controllability of the MF-SDE at a positive initial time implies both
exact controllability at the zero initial time and exact controllability of
the corresponding non-mean-field SDE.

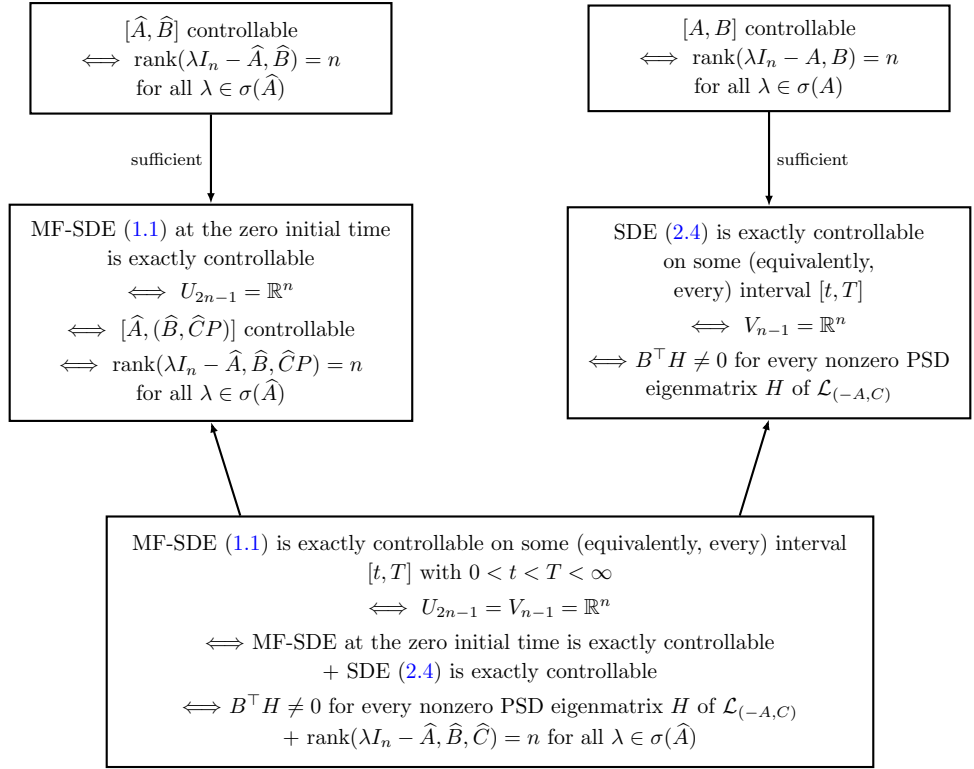
\begin{figure}[htbp]
\centering
\small
\begin{tikzpicture}[
    scale=0.82,
    transform shape,
    top/.style={
        draw,
        rectangle,
        align=center,
        thick,
        inner sep=6pt,
        minimum height=1.0cm,
        text width=5.4cm
    },
    main/.style={
        draw,
        rectangle,
        align=center,
        thick,
        inner sep=7pt,
        minimum height=2.5cm,
        text width=6.0cm
    },
    positive/.style={
        draw,
        rectangle,
        align=center,
        thick,
        inner sep=8pt,
        minimum height=3.0cm,
        text width=11.8cm
    },
    implication/.style={
        -{Latex[length=4pt]},
        thick
    }
]


\node (M) [top] at (-4.5,5.3)
{\mbox{$[\hA,\hB]$ controllable}\\
\mbox{$\Longleftrightarrow\ \rank(\lambda I_n-\hA,\hB)=n$}\\
\mbox{$\text{for all }\lambda\in\sigma(\hA)$}};

\node (O) [top] at (4.5,5.3)
{\mbox{$[A,B]$ controllable}\\
\mbox{$\Longleftrightarrow\ \rank(\lambda I_n-A,B)=n$}\\
\mbox{$\text{for all }\lambda\in\sigma(A)$}};


\node (Z) [main] at (-4.5,1.2)
{MF-SDE \eqref{eq:system} at the zero initial time\\
is exactly controllable\\[2pt]
\mbox{$\Longleftrightarrow\ U_{2n-1}=\dbR^n$}\\[1.5mm]
\mbox{$\Longleftrightarrow\ [\hA,(\hB,\hC P)]$ controllable}\\[1mm]
\mbox{$\Longleftrightarrow\
\rank(\lambda I_n-\hA,\hB,\hC P)=n$}\\
\mbox{$\text{for all }\lambda\in\sigma(\hA)$}};

\node (S) [main] at (4.5,1.2)
{SDE \eqref{eq:system-SDE} is exactly controllable\\
on some (equivalently, every) interval $[t,T]$\\[1mm]
\mbox{$\Longleftrightarrow\ V_{n-1}=\dbR^n$}\\[1mm]
$\Longleftrightarrow$\
\mbox{$B^\top H\ne0$} for every nonzero PSD\\
eigenmatrix $H$ of $\cL_{(-A,C)}$};


\node (P) [positive] at (0,-4.1)
{MF-SDE \eqref{eq:system} is exactly controllable on some
(equivalently, every) interval\\
\mbox{$[t,T]\ \text{with}\ 0<t<T<\i$}\\[1mm]
\mbox{$\Longleftrightarrow\
U_{2n-1}=V_{n-1}=\dbR^n$}\\[1mm]
$\Longleftrightarrow$ MF-SDE at the zero initial time is exactly controllable\\ 
$+$ SDE \eqref{eq:system-SDE} is exactly controllable\\[1mm]
$\Longleftrightarrow$\
\mbox{$B^\top H\ne0$} for every nonzero PSD eigenmatrix $H$ of
$\cL_{(-A,C)}$ \\ 
$+$ \mbox{$\rank(\lambda I_n-\hA,\hB,\hC)=n$}
\mbox{$\text{for all }\lambda\in\sigma(\hA)$}};


\draw[implication] (M.south) --
node[left,pos=0.5]{\scriptsize sufficient}
(Z.north);

\draw[implication] (O.south) --
node[right,pos=0.5]{\scriptsize sufficient}
(S.north);


\draw[implication]
($(P.north west)!0.35!(P.north)$) -- (Z.south);

\draw[implication]
($(P.north east)!0.35!(P.north)$) -- (S.south);

\end{tikzpicture}

\caption{Summary of the main controllability criteria and their relationships.}
\label{fig:summary}
\end{figure}



\begin{thebibliography}{90}
\addtolength{\itemsep}{-1.0ex}

\bibitem{barbu2023}
\rm V.~Barbu,
\it Exact controllability of Fokker-Planck equations and McKean-Vlasov SDEs,
\rm SIAM J. Control Optim., 61 (2023), pp. 1805-1818.

\bibitem{bi2020}
\rm X.~Bi, J.~Sun, and J.~Xiong,
\it Optimal control for controllable stochastic linear systems,
\rm ESAIM Control Optim. Calc. Var., 26 (2020): 98.

\bibitem{buckdahn2006}
\rm R.~Buckdahn, M.~Quincampoix, and G.~Tessitore,
\it A characterization of approximately controllable linear stochastic differential equations,
\rm in Stoch. Partial Differ. Equ. Appl.--VII, Lect. Notes Pure Appl. Math. 245, Chapman \& Hall/CRC, Boca Raton (2006), pp. 53-60.

\bibitem{chen2024}
\rm C.~Chen and Z.~Yu,
\it Exact controllability for mean-field type linear game-based control systems,
\rm Appl. Math. Optim., 90 (2024): 3.

\bibitem{ehrhardt1982}
\rm M.~Ehrhardt and W.~Kliemann,
\it Controllability of linear stochastic systems,
\rm Systems Control Lett., 2 (1982), pp. 145-153.

\bibitem{goreac2008}
\rm D.~Goreac,
\it A Kalman-type condition for stochastic approximate controllability,
\rm C. R. Math. Acad. Sci. Paris, 346 (2008), pp. 183-188.

\bibitem{goreac2026}
\rm D.~Goreac, J.~Li, and X.~Zhang,
\it Controllability concepts for mean-field dynamics with reduced-rank coefficients,
\rm SIAM J. Control Optim., 64 (2026), pp. 2895-2916.

\bibitem{goreac2014}
\rm D.~Goreac,
\it Controllability properties fo linear mean-field stochastic systems,
\rm Stoch. Anal. Appl., 32 (2014), pp. 280-297.

\bibitem{hautus}
\rm M.~L.~J.~Hautus,
\it Controllability and observability conditions of linear autonomous systems,
\rm Proc. Nederl. AKad. Wetensch., Ser. A, 72 (1969), pp. 443-448.

\bibitem{liu2010}
\rm F.~Liu and S.~Peng,
\it On controllability for stochastic control systems when the coefficient is time-variant,
\rm J. Syst. Sci. Complex., 23 (2010), pp. 270-278.

\bibitem{mahmudov2000}
\rm N.~I.~Mahmudov and A.~Denker,
\it On controllability of linear stochastic systems, 
\rm Int. J. Control, 73 (2000), pp. 144-151.

\bibitem{mahmudov2006}
\rm N.~I.~Mahmudov and M.~A.~McKibben,
\it McKean-Vlasov stochastic differential equations in Hilbert spaces under Carathe\'odory conditions,
\rm Dynam. Systems Appl., 15 (2006), pp. 357-374.

\bibitem{park2008}
\rm J.~Y.~Park, P.~Balasubramaniam, and Y.~H.~Kang,
\it Controllability of McKean-Vlasov stochastic integrodifferential evolution equation in Hilbert spaces,
\rm Numer. Funct. Anal. Optim., 29 (2008), pp. 1328-1346.

\bibitem{peng1994}
\rm S.~Peng,
\it Backward stochastic differential equation and exact controllability of stochastic control systems,
\rm Prog. Nat. Sci., 4 (1994), pp. 274-284.





\bibitem{sun2026}
\rm J.~Sun,
\it Hautus-type criteria for controllability and stabilizability of backward-structured stochastic systems,
\rm arXiv:2605.28242v1.

\bibitem{wang2020}
\rm Y.~Wang and Z.~Yu,
\it On the partial controllability of SDEs and the exact controllability of FBSDEs,
\rm ESAIM Control Optim. Calc. Var., 26 (2020): 68.

\bibitem{wang2017}
\rm Y.~Wang, D.~Yang, J.~Yong, and Z.~Yu,
\it Exact controllability of linear stochastic differential equations and related problems,
\rm Math. Control Relat. Fields, 7 (2017), pp. 305-345.


\bibitem{ye2020}
\rm W.~Ye and Z.~Yu,
\it Exact controllability of linear mean-field stochastic systems and observability inequality for mean-field backward stochastic differential equations,
\rm Asian J. Control, 24 (2022), pp. 237-248.


\bibitem{yu2021}
\rm Z.~Yu,
\it Controllability Gramian and Kalman rank condition for mean-field control systems,
\rm ESAIM Control Optim. Calc. Var., 27 (2021): 30.

\bibitem{zabczyk1981}
\rm J.~Zabczyk,
\it Controllability of stochastic linear systems,
\rm Systems Control Lett., 1 (1981), pp. 25-31.

\end{thebibliography}
\end{document}